\documentclass[11pt]{article}

\usepackage{amsfonts,amsmath,amssymb}
\usepackage{amsthm}
\usepackage{graphicx}

\newtheorem{theorem}{Theorem}[section]
\newtheorem{proposition}[theorem]{Proposition}

\newtheorem{conjecture}[theorem]{Conjecture}

\newcommand{\beq}[1]{\begin{equation}\label{#1}}
\newcommand{\enq}[0]{\end{equation}}

\newcommand{\f}[0]{{\cal F}}
\newcommand{\G}[0]{{\cal G}}
\newcommand{\p}[0]{{\cal P}}

\newcommand{\pn}[0]{{\cal P}([n])}

\usepackage[usenames]{color}
\begin{document}

\title{A Note on Large H-Intersecting Families}

\author{
Nathan Keller\thanks{Department of Mathematics, Bar Ilan University, Ramat Gan, Israel.
{\tt nathan.keller27@gmail.com}. Research supported by the Israel Science Foundation (grants no.
402/13 and 1612/17), the Binational US-Israel Science Foundation (grant no. 2014290), and by the Alon Fellowship.}
\mbox{ } and Noam Lifshitz\thanks{Department of Mathematics, Bar Ilan University, Ramat Gan, Israel.
{\tt noamlifshitz@gmail.com}.}
}

\maketitle

\begin{abstract}
A family $\f$ of graphs on a fixed set of $n$ vertices is called \emph{triangle-intersecting} if for any $G_1,G_2 \in \f$, the intersection $G_1 \cap G_2$ contains a triangle. More generally, for a fixed graph $H$, a family $\f$ is $H$-intersecting if the intersection of any two graphs in $\f$ contains a sub-graph isomorphic to $H$.

In~\cite{EFF12}, Ellis, Filmus and Friedgut proved a 36-year old conjecture of Simonovits and S\'{o}s stating that the maximal size of a triangle-intersecting family is $(1/8)2^{n(n-1)/2}$. Furthermore, they proved a $p$-biased generalization, stating that for any $p \leq 1/2$, we have
$\mu_{p}\left(\f\right)\le p^{3}$, where $\mu_{p}\left(\f\right)$ is the probability that the random graph $G\left(n,p\right)$ belongs to $\f$.

In the same paper, Ellis et al. conjectured that the assertion of their biased theorem holds also for $1/2 < p \le 3/4$, and more generally, that for any non-$t$-colorable graph $H$ and any $H$-intersecting family $\f$, we have $\mu_{p}\left(\f\right)\le p^{t(t+1)/2}$ for all $p \leq (2t-1)/(2t)$.

In this note we construct, for any fixed $H$ and any $p>1/2$, an $H$-intersecting family $\f$ of graphs such that $\mu_{p}\left(\f\right)\ge 1-e^{-n^{2}/C}$, where $C$ depends only on $H$ and $p$, thus disproving both conjectures.
\end{abstract}

\section{Introduction}

Denote $[n]=\{1,2,\ldots,n\}$. Throughout the paper, $\G_n$ denotes the family of all graphs on a fixed set of $n$ vertices.

A family $\f\subset\p([n])$ is said to be \emph{intersecting} if for any $A,B\in\f$, $A\cap B\neq\emptyset$. The classical Erd\H{o}s-Ko-Rado (EKR) theorem~\cite{EKR} determines the maximal size of an intersecting family of $k$-element subsets of $[n]$.
\begin{theorem}[Erd\H{o}s, Ko, and Rado, 1961]
\label{Thm:EKR} Let $k<n/2$,
and let $\f\subset[n]^{(k)}$ be an intersecting family. Then $|\f|\leq{n-1 \choose k-1}$.
Equality holds if and only if $\f = \{A \in [n]^{(k)}: x \in A\}$, for some $x \in [n]$.
\end{theorem}
The EKR theorem is the cornerstone of an entire subfield of extremal combinatorics called `intersection problems for finite sets', which studies how large can a family of sets be, given some intersection constraints on its elements. See~\cite{FT16} for a recent survey of the topic.

\medskip

\noindent Along with intersection problems on families of $k$-element sets (called $k$-uniform families), it is quite common to consider $p$-\emph{biased} versions of the problems, in which one wants to find the maximal $p$-biased measure of a family $\f\subseteq\pn$, defined by $\mu_{p}(\f):=\sum_{S\in\f}p^{|S|}(1-p)^{n-|S|}$, given that $\f$ satisfies some intersection constraints. In this setting, the basic result is the $p$-biased version of the EKR Theorem, proved by Ahlswede and Katona~\cite{Ahlswede Katona} in 1977, which asserts that for any intersecting $\f$ and any $p\le\frac{1}{2}$, we have $\mu_{p}\left(\f\right)\le p$. Biased intersection theorems usually follow from the corresponding $k$-uniform results (see~\cite{Dinur Safra}). In the other direction,  in some cases $k$-uniform results were deduced from their $p$-biased analogues (see, e.g.,~\cite{friedgut2008measure,lee2016towards}). For a semi-random sample of recent $p$-biased intersection results, see~\cite{EKL16+,frankl2014erdHos,lee2016towards} and the references therein.

\medskip

One of the best-known intersection problems was determining the maximal size of a triangle-intersecting family of graphs. In 1976, Simonovits and S\'{o}s conjectured that the maximum is attained by the family of all graphs that contain a fixed triangle, and thus, $|\f| \leq (1/8)2^{n(n-1)/2}$ for any triangle-intersecting $\f \subset \G_n$. The first major step toward resolution of the conjecture was made in 1986, when Chung et al.~\cite{CGFS86} showed that $|\f| \leq (1/4)2^{n(n-1)/2}$, using entropy methods. It took 26 more years until the Simonovits-S\'{o}s conjecture was proved in a beautiful paper of Ellis, Filmus and Friedgut~\cite{EFF12} using spectral methods and Fourier analysis. Actually, Ellis et al. proved the following more general biased version of the conjecture:
\begin{theorem}[Ellis, Filmus, and Friedgut, 2012]\label{Thm:EFF}
Let $\f$ be a triangle-intersecting family of graphs. Then for any $p \leq 1/2$, we have $\mu_p(\f) \leq p^3$.
\end{theorem}
Note that the Simonovits-S\'{o}s conjecture is the case $p=1/2$ of Theorem~\ref{Thm:EFF}. Ellis et al. also proved several extensions and variants of the theorem, including a version for odd-cycle intersecting families and a $k$-uniform version for any $k = \alpha {{n}\choose{2}}$, $\alpha<1/2$.

\medskip

In the same paper, Ellis et al. conjectured that Theorem~\ref{Thm:EFF} holds in much larger generality.
\begin{conjecture}[\cite{EFF12}]
\label{Conj:EFF}
Let $H$ be a non-$t$-colorable graph. Then for any $H$-intersecting family $\f$ and any $p \leq \frac{2t-1}{2t}$, we have
\[
\mu_p(\f) \leq p^{ {{t+1}\choose{2}}},
\]
and the maximum is attained if and only if $H$ is the complete graph on $t+1$ vertices.
\end{conjecture}
Note that one can easily show, using the classical Tur\'{a}n's theorem, that for $p> \frac{2t-1}{2t}$ there exist $K_{t+1}$-intersecting families of measure $1-o(1)$. Thus, Conjecture~\ref{Conj:EFF} is the strongest result one may hope for.

\medskip

In this note we prove the following result, which disproves Conjecture~\ref{Conj:EFF}.
\begin{proposition}\label{Main}
For any graph $H$ and any $p>\frac{1}{2}$, there exists an $H$-intersecting graph family $\f \subset \G_n$
such that $\mu_{p}\left(\f\right)\ge 1-e^{-n^{2}/C}$, where $C=C\left(p,H\right)>0$.
\end{proposition}
The proof of Proposition~\ref{Main} is rather elementary -- we construct inductively a sequence of families $\{\f_t\}_{t=2,3,\ldots}$, such that each $\f_t$ is $K_t$-intersecting, and show that $\mu_{p}\left(\f_t\right)$ satisfies the assertion of the theorem using classical Chernoff bounds~\cite{AS}. The idea behind the proof is the fact that for any fixed $p>1/2$, any fixed graph $H$, and a sufficiently large $n$, two `generic' random graphs on $n$ vertices of edge density $p$ are $H$-intersecting. Hence, one may expect that if a family $\f$ contains only `sufficiently pseudo-random' graphs then it is $H$-intersecting. We show that one can indeed construct such a family, such that its $\mu_p$ measure will be close to $1$.

\medskip \noindent Note that by the biased EKR theorem mentioned above, for any non-empty $H$ and any $p \leq 1/2$, any $H$-intersecting family $\f$ satisfies $\mu_p(\f) \leq p$ (and in particular, there do not exist $H$-intersecting families of $p$-measure close to 1). Hence, our result implies that the maximal $p$-measure of an $H$-intersecting family exhibits a \emph{sharp threshold} phenomenon at $p=1/2$. It may be interesting to further understand the `threshold window', and in particular, to determine the maximal $p=p(n,H)$ such that for any $H$-intersecting family $\f$, $\mu_p(\f)$ is bounded away from 1 (for some fixed graph $H$).

\section{Proof of Proposition~\ref{Main}}
\label{sec:proof}

We use the following standard consequence of Chernoff's inequality (see~\cite{AS}, Appendix~A).
\begin{proposition}
\label{Prop:Chernoff}
For any $N \geq 1$ and $p'<p<1$, there exists a constant $C=C\left(p,p'\right)>0$ such that the following holds. Let $X\sim\mathrm{Bin}\left(N,p\right)$.
Then
\[
\Pr\left[X\le p'\right]\le e^{-N/C}.
\]
\end{proposition}

\begin{proof}[Proof of Proposition~\ref{Main}]
It is clearly sufficient to prove the proposition for all complete graphs $H = K_{t}$, $t \in \mathbb{N}$. We prove the proposition by induction on $t$, by constructing (for each $t$) a $K_t$-intersecting family $\f_t^n \subset \G_n$ that satisfies the assertion of the proposition.
Recall that for any family $\f \subset \G_n$, and for any $p$, $\mu_p(\f)$ is the probability that a random graph $G \sim G(n,p)$ belongs to $\f$.

\medskip \noindent For $t=2$ and for any $n \in \mathbb{N}$, we define $\f_{2}^n \subset \G_n$ as the family of all graphs that contain more than half of the ${{n}\choose{2}}$ possible edges. $\f_2$ is clearly $K_{2}$-intersecting, and by Proposition~\ref{Prop:Chernoff}, we have
\[
\mu_{p}\left(\f_2^n\right)\ge 1- e^{-n^{2}/C},
\]
where $C=C\left(p\right)$, as asserted. (Note that the number of edges in $G \sim G(n,p)$ has distribution $\mathrm{Bin}\left(\binom{n}{2},p\right)$, and thus we indeed can apply Proposition~\ref{Prop:Chernoff} to bound $\Pr_{G \sim G(n,p)} [G \in \f_2^n]$.)

\medskip \noindent Suppose that we already defined $K_{t}$-intersecting families $\f_{t}^m \subset \G_m$ (for all $m \in \mathbb{N}$) such that $\mu_{p}\left(\f_{t}^m\right) \ge 1-e^{-m^{2}/C\left(p,t\right)}$. For any $n \in \mathbb{N}$, we define $\f_{t+1}^n \subset \G_n$ to be the family of all graphs $G$ such that:
\begin{enumerate}
\item $G$ has at least $\frac{p+0.5}{2}\binom{n}{2}$ edges.
\item For every subset $S\subseteq\left[n\right]$ with $|S| \geq \left(p-\frac{1}{2}\right)(n-1)$, the induced sub-graph of $G$ on the vertex set $S$ (denoted by $G|_S$) belongs to $\f_{t}^{|S|}$.
\end{enumerate}
Let $G \sim G\left(n,p\right)$. By Proposition~\ref{Prop:Chernoff}, we have
\[
\Pr[G \mbox{ satisfies (1) }] \geq 1-e^{-n^{2}/C\left(p\right)}.
\]
In addition, for any fixed $S \subset [n]$ with $|S| \geq \left(p-0.5\right)(n-1)$, we have
\[
\Pr \left[ G \big|_S \in \f_t^{|S|} \right] \geq 1- e^{-\left(p-0.5\right)^{2}(n-1)^{2}/C\left(p,t\right)}
\]
by the induction hypothesis. Hence, a union bound implies
\begin{align*}
\mu_{p}\left(\f_{t+1}^n\right) \geq 1 - e^{-n^{2}/C\left(p\right)} - \sum_{S\subseteq\left[n\right]}e^{-\left(p-0.5\right)^{2}(n-1)^{2}/C\left(p,t\right)} \geq 1-e^{-n^{2}/C\left(t+1,p\right)}.
\end{align*}
 %& \ge1-\sum_{S\subseteq\left[n\right]}e^{-\frac{\left(p-0.5\right)^{2}n^{2}}{C\left(p,t\right)}}+ne^{-\frac{n^{2}}{C\left(p\right)}}\\
%& =1-2^{n}e^{-\frac{\left(p-0.5\right)^{2}n^{2}}{C\left(p,t\right)}}+ne^{-\frac{n^{2}}{C\left(p\right)}}\\
%& =1-e^{-C\left(t+1,p\right)n^{2}}
%\end{align*}
We assert that $\f_{t+1}^n$ is $K_{t+1}$-intersecting. To prove this, let $G_1, G_2 \in\f_{t+1}^{\left(n\right)}$, and let $G_0 = G_1\cap G_2$. We show that $G_0$ contains a copy of $K_{t+1}$.

Let $v$ be a vertex of maximal degree in $G_0$. As $|E(G_0)| \ge |E\left(G_1\right)|+|E\left(G_2\right)|-\binom{n}{2}\ge\left(p-0.5\right)\binom{n}{2}$, we have $\deg\left(v\right) \ge \left(p-0.5\right)(n-1)$. Let $T$ be the set of neighbors of $v$ in $G_0$, and note that $|T| \geq \left(p-0.5\right)(n-1)$. It is clearly sufficient to show that the induced sub-graph $(G_0)|_T$ contains a copy of $K_t$. Consider the induced sub-graphs $(G_1)|_T, (G_2)|_T,$ and $(G_0)_T$. By assumption, we have $(G_1)|_T, (G_2)_T \in \f_t^{|T|}$. Since $\f_t^{|T|}$ is $K_t$-intersecting, this implies that $(G_0)_T = (G_1)|_T \cap (G_2)|_T$ contains a copy of $K_t$. This completes the proof.
%and $H\left(S\right)$ are in $\f_{t}$. Thus $F\left(S\right)=G\left(S\right)\cap H\left(S\right)$
%contains a $t$-clique, and hence $F$ contains a $t+1$ clique as
%desired.
\end{proof}

%\bibliographystyle{plain}
%\bibliography{C:/Users/noam/Desktop/Refs/Refs}

\end{document}